\documentclass[11pt]{article}
\usepackage[T1]{fontenc}
\usepackage[english]{babel}
\usepackage{a4wide}

\usepackage{amsfonts}
\usepackage{amsmath}
\usepackage{amssymb}
\usepackage{amstext}
\usepackage{amsthm}
\usepackage{bm}
\usepackage{caption}
\usepackage{cleveref}
\usepackage{color}
\usepackage{colordvi}
\usepackage{comment}
\usepackage{dutchcal}
\usepackage{enumitem}
\usepackage{float}
\usepackage{graphicx}
\usepackage{graphics}
\usepackage{standalone}
\usepackage{mathtools}
\usepackage{soul}
\usepackage{svg}
\usepackage{tikz}
\usepackage{xcolor}
\usepackage{xspace}
\usepackage[ruled,vlined,noend]{algorithm2e}

\usepackage[textsize=tiny,backgroundcolor=white,linecolor=black] {todonotes} 

\DeclareCaptionLabelFormat{andtable}{#1~#2  \&  \tablename~\thetable}

\DeclareSymbolFont{sfoperators}{OT1}{cmss}{m}{n}
\DeclareSymbolFontAlphabet{\mathsf}{sfoperators}
\makeatletter
\def\operator@font{\mathgroup\symsfoperators}
\makeatother

\DeclareMathOperator{\bpg}{BPG}
\DeclareMathOperator{\chain}{CHAIN}
\DeclareMathOperator{\conv}{CONV}

\crefname{enumi}{point}{points} 
\Crefname{enumi}{Point}{points} 

\newlist{myenum}{enumerate}{3}
\setlist[myenum,1]{label=\textbf{\alph*.},
                   ref  =\alph*}
\setlist[myenum,2]{label=\bfseries(\roman*),
                   ref  =\themyenumi\textbf{.(\roman*)}}

\crefname{myenumi}{point}{points}
\crefname{myenumii}{point}{points}
\Crefname{myenumi}{Point}{Points}
\Crefname{myenumii}{Point}{Points}

\def\square{\vbox{\hrule height.2pt\hbox{\vrule width.2pt height5pt \kern5pt
\vrule width.2pt} \hrule height.2pt}}

        {\hspace*{\fill}$\Box$\par\vspace{4mm}}

\newtheorem{lemma}{Lemma}
\newtheorem{theorem}{Theorem}
\newtheorem{observation}{Observation}

\crefname{enumi}{point}{points} 
\Crefname{enumi}{Point}{Points} 

\crefname{algocf}{algorithm}{algorithms} 
\Crefname{algocf}{Algorithm}{Algorithm}

\title{A Freeable Matrix Characterization of Bipartite Graphs of Ferrers Dimension Three}

\author{
Parinya Chalermsook \thanks{The University of Sheffield, UK. E-mail: {\tt chalermsook@gmail.com}.}
\and
Ly Orgo \thanks{Aalto University, Finland. E-mail: {\tt ly.orgo@aalto.fi}.}
\and
Minoo Zarsav \thanks{Aalto University, Finland. E-mail: {\tt minoo.zarsav@aalto.fi}.}
}

\makeatletter

\newcommand{\shorttitle}{\@title}

\begin{document}

\maketitle
\begin{abstract}
Ferrers dimension, along with the order dimension, is a standard dimensional concept for bipartite graphs.  In this paper, we prove that a graph is of Ferrers dimension three (equivalent to the intersection bigraph of orthants and points in ${\mathbb R}^3$) if and only if it admits a biadjacency matrix representation that does not contain $\Gamma= \begin{psmallmatrix}
            * & 1 & * \\
            1 & 0 & 1 \\
            0 & 1 & * 
            \end{psmallmatrix} \mbox{ and }
            \Delta = \begin{psmallmatrix}
            1 & * & *  \\
            0 & 1 & * \\
            1 & 0 & 1
            \end{psmallmatrix}$, where $*$ denotes a zero or one entry. 
\end{abstract}

\section{Introduction}

Special graph classes (or structured graphs) have played crucial roles in combinatorics and computer science, arising in many application domains. For instance, in the area of algorithm designs, special graph classes have been used, in the past decades, to model real-world data sets. 
These graphs are often equipped with geometric, topological, or combinatorial properties that can be computationally leveraged. 

In this paper, we consider the class of bipartite graphs that have Ferrers dimension three, denoted by $\chain^3$ (see Section~\ref{sec:prelim} for the definition of Ferrers dimension). In a geometric language, it is equivalent to the intersection bigraph of \textbf{3D orthants} and \textbf{points} in ${\mathbb R}^3$: A bipartite graph $G = (A \cup B,E)$ is of Ferrers dimension three if and only if each vertex $a \in A$ is associated with a point $p_a \in {\mathbb R}^3$ and each vertex $b \in B$ with an orthant $O_b = (-\infty, x_b) \times (-\infty, y_b) \times (-\infty, z_b)$ such that $(a,b) \in E$ if and only if $p_a \in O_b$. The class of 3D orthants arise naturally in data structures (see, e.g.,~\cite{chan2011orthogonal}) and more recently in extremal combinatorics~\cite{DBLP:conf/soda/ChanH23,GD-accept}. 
Ferrers dimension is a standard dimensionality concept closely related to order dimension and interval dimension~\cite{felsner1994interplay}. Well-known classes of bipartite graphs (e.g., orthogonal ray graphs~\cite{chaplick2014intersection} and grid intersection graphs~\cite{hartman1991grid}) have been shown to have constant Ferrers dimension. 

It is common for a graph class to admit several equivalent characterizations.
The main contribution of this paper is a characterization of $\chain^3$ in terms of a freeable matrix property. Let us start by defining the terminologies.
Let $P$ be a 0/1 matrix. We say that matrix $M$ \textbf{contains} $P$ if a submatrix of $M$ is equal to $P$; otherwise, when $M$ does not contain $P$, we say that $M$ is $P$-\textbf{free}. We often refer to submatrix $P$ in this role as a \textbf{pattern}. 

When we allow the entries of $P$ to be $\{0,1,*\}$, a star can represent either zero or one. 
We say that matrix $M'$ \textbf{matches with} $P$ if for all row $i$ and column $j$, we have $M'[i,j] = P[i,j]$ or $P[i,j]=*$. 
Extending the notion of containment, matrix $M$ contains $P$ if a submatrix of $M$ matches with $P$, so $M$ is $P$-free when it is free of all matrices $P'$ obtained by replacing each $*$ of $P$ by either zero or one. 
A graph $G$ is said to be $P$-\textbf{freeable} if there exists a biadjacency matrix representation of $G$ that is $P$-free. 

Many (geometric) bipartite graph classes are known to admit both intersection bigraph representation and freeable matrix characterization. Table~\ref{fig:freeable_overview} summarizes the existing results for graph classes in the context of our work.

\begin{table}[h]
\centering
    \begin{tabular}[b]{lll}
    \hline 
    \textbf{Graph classes} & \textbf{Freeable Pattern} & \textbf{References}  \\ \hline
        $\chain$ &  $\begin{psmallmatrix}
            0 & 1 
            \end{psmallmatrix}$  & \cite{das1989interval} \\ 
        $\bpg$ & $\begin{psmallmatrix}
            1 & 0 \\
            * & 1 
            \end{psmallmatrix}$ $\begin{psmallmatrix}
            1 & * \\
            0 & 1 
            \end{psmallmatrix}$  & \cite{chen1993efficient} \\ 
        $\conv$ &  $\begin{psmallmatrix}
            1 & 0 & 1
            \end{psmallmatrix}$  & \cite{brandstfidt1991special} 
            \\ 
        $\chain^2$ &  $\begin{psmallmatrix}
            1 & * \\
            0 & 1 
            \end{psmallmatrix}$  & \cite{shrestha2010orthogonal} \\ 
        Chordal Bipartite &  $\begin{psmallmatrix}
            1 & 1 \\
            0 & 1 
            \end{psmallmatrix}$  & \cite{klinz1992permuting} \\ 
        Stick Graphs &  $\begin{psmallmatrix}
            * & 1 & *\\
            1 & 0 & 1
            \end{psmallmatrix}$ \& 
            $\begin{psmallmatrix}
            1 & *\\
            0 & 1\\
            1 & *
            \end{psmallmatrix}$ \& 
            $\begin{psmallmatrix}
            * & 1 & *\\
            * & 0 & 1\\
            1 & * & *
            \end{psmallmatrix}$   & \cite{de2019recognition} \\ 
        Segment Ray &  $\begin{psmallmatrix}
            * & 1 & *\\
            1 & 0 & 1
            \end{psmallmatrix}$  & \cite{chaplick2014intersection} \\ 
        Grid Intersection (GIG) &  $\begin{psmallmatrix}
            * & 1 & *\\
            1 & 0 & 1\\
            * & 1 & *
            \end{psmallmatrix}$  & \cite{hartman1991grid} \\ 
        $\mathbf{\chain^3}$ &  $\begin{psmallmatrix}
            * & 1 & * \\
            1 & 0 & 1 \\
            0 & 1 & * 
            \end{psmallmatrix}$ \& 
            $\begin{psmallmatrix}
            1 & * & *  \\
            0 & 1 & * \\
            1 & 0 & 1
            \end{psmallmatrix}$   & this paper \\ \hline
    \end{tabular}
\qquad
\includegraphics[scale=0.14]{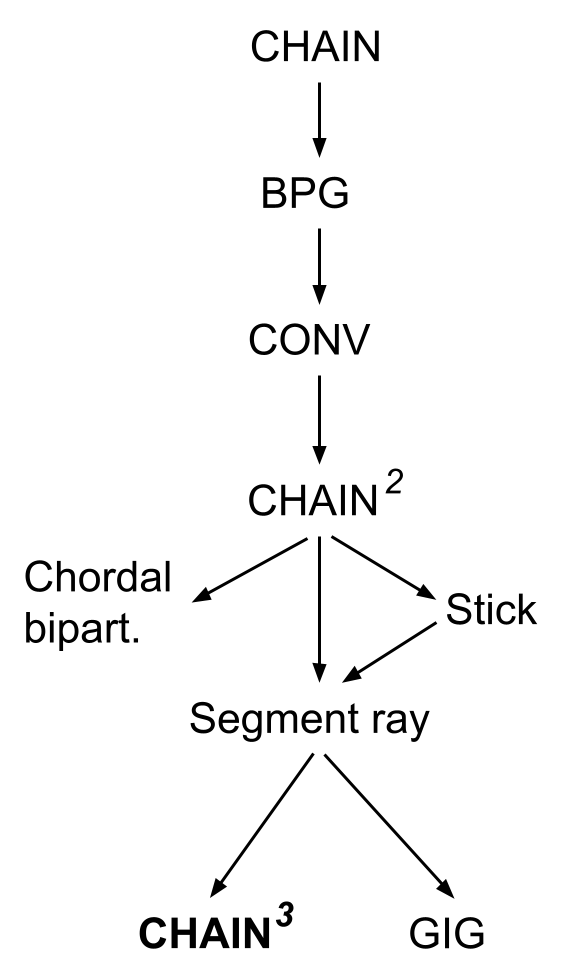}
\captionlistentry[table]{Table}
\caption{Freeable matrix characterizations and (right-hand-side) relations between the corresponding graph classes. The arrow from  graph class $X$ to $Y$ denotes that $X$ is a subclass of $Y$. The graph class $\chain^d$ contains all graphs of Ferrers dimension $d$. \label{fig:freeable_overview}}
\end{table}

\begin{theorem}
\label{thm: main}
A bipartite graph $G$ has Ferrers dimension three if and only if it is $\Gamma$ and $\Delta$-freeable where $\Gamma= \begin{psmallmatrix}
            * & 1 & * \\
            1 & 0 & 1 \\
            0 & 1 & * 
            \end{psmallmatrix} \mbox{ and }
            \Delta = \begin{psmallmatrix}
            1 & * & *  \\
            0 & 1 & * \\
            1 & 0 & 1
            \end{psmallmatrix}$   
\end{theorem}

\section{Preliminaries}
\label{sec:prelim}


A bipartite graph $G=(U\cup V, E)$ is a {\bf chain graph} ($\chain$), if its vertices can be linearly ordered as $U=\{u_1, u_2, \ldots, u_{|U|} \}$ and $V=\{v_1, v_2, \ldots, v_{|V|} \}$ so that we have the chains of neighbors, $N(u_1)\subseteq N(u_2)\subseteq \ldots\subseteq N(u_{|U|})$ and $N(v_1)\subseteq N(v_2)\subseteq \ldots\subseteq N(v_{|V|})$. Chain graphs are also called difference graphs, Ferrers bigraphs, and induced $2K_2$-free graphs \cite{hartman1991grid}. 
The graph class is exactly the \textbf{intersection bigraph of rays and points} in ${\mathbb R}$ (that is, it is representable as points and rays on a real line, such that there is an edge $\{u, v\}$ for some $u \in U$, $v \in V$ if and only if the ray representing $v$ contains the point representing $u$ \cite{chaplick2014intersection}.) 

\begin{lemma}[\cite{das1989interval}]
\label{thm: chain freeable}
Chain graphs are equivalent to graphs whose biadjacency matrices are $\begin{psmallmatrix} 1 & 0\end{psmallmatrix}$-freeable.  
\end{lemma}

Moreover, chain graphs do not contain an induced matching of size two. In the forbidden matrix terminology, this can be stated as follows: 

\begin{lemma}[\cite{hartman1991grid}]
\label{thm: chain induced matching}  
Let $A$ be a biadjacency matrix of a chain graph (any orderings of rows and columns). Then $A$ does not contain pattern $\begin{psmallmatrix}
1 & 0  \\
0 & 1  
\end{psmallmatrix}$. 
\end{lemma}


For two graphs $G_1 = (V, E_1)$ and $G_2 = (V, E_2)$ on the same set of vertices, the \textit{intersection} $G_1\cap G_2$ is defined as $(V, E_1\cap E_2)$. For positive integer $d$, we use $\chain^d$ to denote the class of all graphs $G$ that can be written as an intersection of $d$ chain graphs, i.e., there exist, for all $i \in [d]$, graph $G_i \in \chain$ such that $G=\bigcap_{i\in [d]} G_i$.
\textbf{Ferrers dimension} of a bipartite graph $G$ is the minimum $d$, such that $G \in \chain^d$ \cite{hartman1991grid}.

\begin{lemma}[\cite{shrestha2010orthogonal}]
\label{thm: chain2 freeable}
Bipartite graph $G$ is in $\chain^2$ if and only if $G$ is $\begin{psmallmatrix}
            1 & * \\
            0 & 1 
            \end{psmallmatrix}$-freeable.  
\end{lemma}



Given a matrix $A$, we write $A[i,j]$ to denote the entry at the $i$-th row and $j$-th column. Let $A$ and $B$ be two matrices of the same size. The \emph{Hadamard product} of $A$ and $B$, denoted by $A\odot B$, is defined by the entry-wise product of the corresponding entries. 

\begin{observation}
For two bipartite graphs $G_1$, $G_2$ on the same vertex set $U \cup V$ and fixed orderings of vertices, the Hadamard product of their biadjacency matrices is a biadjacency matrix of $G_1 \cap G_2$ with respect to the fixed orderings.
\end{observation}

We write $A \leq B$ for matrices $A$ and $B$ if $A[i,j] \leq B[i,j]$ for all row $i$ and column $j$. Notice that, if $C = A \odot B$, then $C \leq A$ and $C \leq B$.

\section{Proof of Theorem~\ref{thm: main}}


\subsection{``Only if'' direction}
Let $G = (U \cup V,E)$. Since $G \in \chain^3$, we have that $G= G_1 \cap G_2 \cap G_3$ where each $G_i \in \chain$. 
Our goal is to show that there exists an ordering $L_U$ of vertices in $U$ and $L_V$ of vertices in $V$ such that the corresponding biadjacency matrix of $G$ is free of the claimed patterns.  
Since $G'= G_1 \cap G_2 \in \chain^2$, we define $L_U$ and $L_V$ to be the orderings of vertices that are guaranteed to exist from~\Cref{thm: chain2 freeable}. 
In particular, the biadjacency matrix $A'$ of $G'$ with rows and columns ordered by $L_U$ and $L_V$ is $D$-free where $D=\begin{psmallmatrix}
            1 & * \\
            0 & 1 
            \end{psmallmatrix}$. 
Let $A_3$ be the biadjacency matrix of $G_3$ whose rows and columns are ordered according to the orderings $L_U$ and $L_V$. 
Now we claim that the Hadamard product $A = A' \odot A_3$ is free of $\Gamma$ and $\Delta$.

\subsubsection*{Pattern $\Gamma$}

Let us assume by contradiction that $A$ contains $\Gamma$. Let the indices of these columns in $A$ be $j_1 < j_2 < j_3$ and the indices of the rows be $i_1 < i_2 < i_3$ such that the submatrix of $A$ in rows $\{i_1, i_2, i_3\}$ and columns $\{j_1, j_2, j_3\}$ matches with $\Gamma$. 
This implies that the submatrix of $A$ at rows $\{i_1, i_2\}$ and columns $\{j_2, j_3\}$ matches with $D$; therefore, $A'[i_1,j_2] = A'[i_2, j_3] = 1$. 
Since $D$ is forbidden in $A'$ (yellow highlight on the left in \cref{fig:freeable_free}), it must be the case that $A'[i_2, j_2]= 1$ (otherwise, $A'$ contains $D$), which implies that $A_3[i_2, j_2]=0$. 

We apply similar reasoning to the submatrix of $A$ at rows $\{i_2,i_3\}$ and columns $\{j_1, j_2\}$ and deduce that $A_3[i_3, j_1] = 0$ (blue highlight on the left in \cref{fig:freeable_free}). Since $A[i_2,j_1] = A[i_3,j_2] = 1$, we must have that $A_3[i_2,j_1] = A_3[i_3,j_2] = 1$. 
Therefore, the pattern $\begin{psmallmatrix}
1 & 0  \\
0 & 1  
\end{psmallmatrix}$  matches with  the submatrix of $A_3$ at the rows $\{i_2, i_3\}$ and columns $\{j_1,j_2\}$. 
This contradicts the fact that $G_3$ is a chain graph (in particular~\Cref{thm: chain induced matching}). 

\subsubsection*{Pattern $\Delta$}

The fact that $A$ avoids $\Delta$ can be argued similarly.  Let the indices of these columns in $A$ be $j_1 < j_2 < j_3$ and the indices of the rows be $i_1 < i_2 < i_3$ such that the submatrix of $A$ at the rows $\{i_1, i_2, i_3\}$ and columns $\{j_1, j_2, j_3\}$ matches with $\Delta$, which implies that the submatrix at $\{i_1,i_2\}$ and $\{j_1,j_2\}$ matches with $D$; therefore $A'[i_1,j_1]= A'[i_2,j_2]=1$. Since $D$ is forbidden in $A'$, we must have $A'[i_2,j_1]=1$ which implies $A_3[i_2,j_1]=0$. 

We apply similar reasoning to the submatrix of $A$ at rows $\{i_2,i_3\}$ and columns $\{j_2, j_3\}$ and deduce that $A_3[i_3, j_2] = 0$ (blue highlight on the right figure in \cref{fig:freeable_free}).
Since $A[i_2,j_2] = A[i_3,j_1] = 1$, we have $A_3[i_2,j_2]=  A_3[i_3,j_1]=1$. Altogether, we have the submatrix of $A_3$ at the rows $\{i_2,i_3\}$ and columns $\{j_1,j_2\}$ matches with the pattern $\begin{psmallmatrix}
1 & 0  \\
0 & 1  
\end{psmallmatrix}$, a contradiction to~\Cref{thm: chain induced matching}.

\begin{figure}[h]
    \centering
    \includegraphics[scale=0.25]{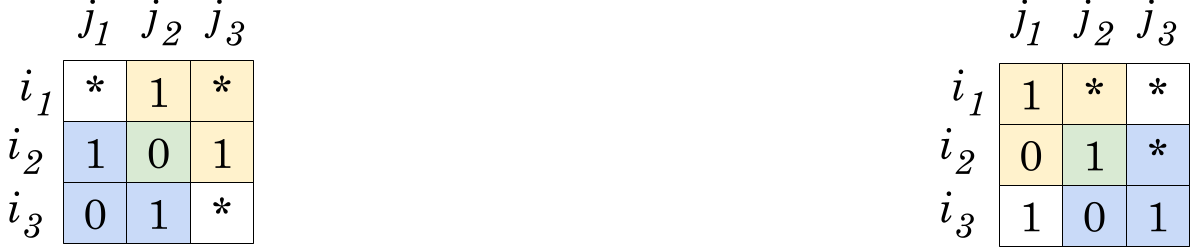}
    \caption{$A'$ cannot contain $D$ as a submatrix of $\Gamma$ or $\Delta$ on the left and right respectively.}
    \label{fig:freeable_free}
\end{figure}

\subsection{``if'' Direction}


Given a graph $G = (U \cup V, E)$ with a biadjacency matrix $A$ that does not contain the submatrices $\Gamma$ and $\Delta$, we will show that $G \in \chain^{3}$. 
Let $L_U$ and $L_V$ be the orderings of $U$ and $V$ according to rows and columns of $A$. 

\begin{lemma} \label{maximal_chain2_submatrix}
Given a graph $G = (U \cup V, E)$ with a biadjacency matrix $A$, that avoids $\Gamma$ or $\Delta$ as a submatrix, there exists a graph $G_{1,2} \in \chain^2$ with a biadjacency matrix $A_{1,2}$ such that 
\begin{itemize}
    \item[(a)] \label{maximal_chain2_submatrix_a}
    $A_{1,2}$ is free from the submarix
        $\begin{psmallmatrix}
         1 & * \\
        0 & 1 
        \end{psmallmatrix}$,
    \item[(b)] $A \leq A_{1,2}$ 
    \item[(c)] Define $\widetilde{A}$ by replacing each $0$-entry in $A$ with $0'$ if the corresponding entry of $A_{1,2}$ is zero; and with $0^*$ otherwise (so the entries in $\widetilde{A}$ are in $\{1,0',0^*\}$.) Then every $0^*$ in $\widetilde{A}$ is part of some submatrix $\begin{psmallmatrix} 
         1 & * \\
        0^* & 1 
        \end{psmallmatrix}$ of $\widetilde{A}$. 
\end{itemize}
\end{lemma}
\begin{proof}
We construct graph $G_1 \in \chain$ by describing its point-ray representation: $V = \{v_1, \ldots v_{|V|}\}$ corresponds to points $P^{(1)} = \{p_1, \ldots, p_{|V|}\}$ placed from left to right; $U$ corresponds to the leftward rays $R^{(1)}=\{r_1,\ldots, r_{|V|}\}$. 
For each $i \in [|V|]$, place the starting point of $r_i$ between $p_j$ and $p_{j+1}$ (if it exists), where $j$ is the maximum integer for which $A[i,j]= 1$. 
Let $A_1$ be the biadjacency matrix of $G_1$ where rows are ordered according to the rays, while columns are ordered according to the points. Observe that $A_1 \geq A$ by construction. 

Analogously, we construct $G_2 \in \chain$; let $U = \{u_1, \ldots u_{|U|}\}$ correspond to points $P^{(2)} = \{p'_1, \ldots, p'_{|U|}\}$ placed from top to bottom; $V$ correspond to the downward rays, where each ray $r'_j$ starts between $p'_{i-1}$(if it exists) and $p'_{i}$ so that $i$ is the least index where $A[i,j]=1$. Let $A_2$ be the biadjacency matrix of $G_2$ where columns are ordered according to the rays, while rows are ordered according to the points. Similarly to before, $A_2 \geq A$.

Define $A_{1,2} = A_1 \odot A_2$ which is a biadjacency matrix representation of $G_{1,2} = G_1 \cap G_2$. 
Since the biadjacency matrices of $G_1$ and $G_2$ are respectively $\begin{psmallmatrix}
0 & 1 
\end{psmallmatrix}$- and $\begin{psmallmatrix}
1 \\
0
\end{psmallmatrix}$-free, then
 $A_{1,2}$ is free from submatrix
        $\begin{psmallmatrix}
         1 & * \\
        0 & 1 
        \end{psmallmatrix}$ which proves point (a) of \Cref{maximal_chain2_submatrix_a}.


Point (b) follows from the fact that $A_1 \geq A$ and $A_2 \geq A$. 

Now let us consider the entry $\widetilde{A}[i,j] = 0^*$. 
Based on the definition, we have that $A[i,j] = 0$ while $A_{1,2}[i,j] =1$. Let $j_{\max}$ be the maximum column such that $A[i,j_{\max}] = 1$; notice that $j_{\max} >j$ since $A_1[i,j]=1$ while $A[i,j]=0$. Similarly, let $i_{\min}$ be the minimum row such that $A[i_{\min}, j] =1$; again, we have that $i_{\min} <i$. 
This implies that the submatrix of $\widetilde{A}$ at rows $\{i_{\min}, i\}$ and $\{j, j_{\max}\}$ matches with pattern $\begin{psmallmatrix} 
         1 & * \\
        0^* & 1 
        \end{psmallmatrix}$. 
This completes the proof of  (c). 
\end{proof}


\begin{figure}
    \centering
    \includegraphics[width=0.9\linewidth]{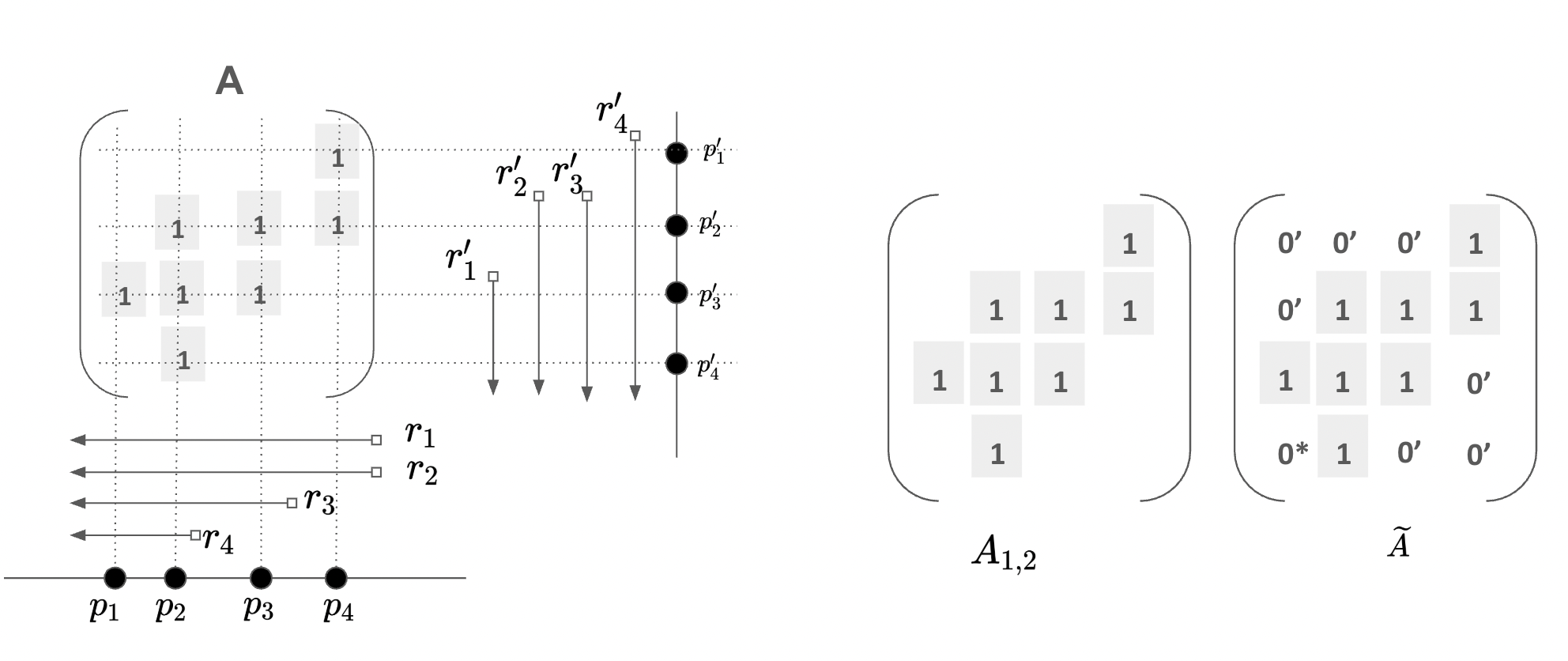}
    \caption{A construction of $A_{1,2}$ from matrix $A$. The ray-point representations of $G_1$ and $G_2$ are shown below and to the right of matrix $A$ respectively. The zero entries in the matrices are left blank. Notice that $A_{1,2}[4,1] \neq A[4,1]$ but other entries are equal. }
    \label{fig:construction}
\end{figure}

\begin{lemma} \label{A_free_matrices}
    The  matrix $\widetilde{A}$ does not include any of the following matrices as a submatrix: 
    $\begin{psmallmatrix}
 0^* & 1 \\
1 & 0^* 
\end{psmallmatrix}$, $\begin{psmallmatrix}
 1 & 0^* \\
0^* & 1 
\end{psmallmatrix}$ , $\begin{psmallmatrix}
 1 & * \\
0' & 1 
\end{psmallmatrix}$
, $\begin{psmallmatrix}
 0^* & * \\
0' & 0^* 
\end{psmallmatrix}$
\end{lemma}

\begin{proof}
 If $\widetilde{A}$ includes 
 $\begin{psmallmatrix}
 1 & 0^* \\
0^* & 1 
\end{psmallmatrix}$ 
 as a submatrix, then according to \Cref{maximal_chain2_submatrix}, there  exists a $1$-entry to the right and above the $0^*$. 
These $1$-entries with the submatrix 
 $\begin{psmallmatrix}
 1 & 0^* \\
0^* & 1 
\end{psmallmatrix} $, implies that $\widetilde{A}$ contains  $\begin{psmallmatrix}
 * & 1 & * \\ 
 1 & 0^* & 1 \\
0^* & 1 & * 
\end{psmallmatrix}$ and therefore $A$ must contain pattern $\Gamma$, which is a contradiction.
With similar reasoning,  $\widetilde{A}$ is  $\begin{psmallmatrix}
 0^* & 1 \\
1 & 0^* 
\end{psmallmatrix}$-free, or otherwise, $A$ contains pattern $\Delta$.

Recall that $A_{1,2}$ does contain  $\begin{psmallmatrix}
 1 & * \\
0 & 1 
\end{psmallmatrix}$. 
If $\widetilde{A}$ contains $\begin{psmallmatrix}
 1 & * \\
0' & 1 
\end{psmallmatrix}$, by our definition, we have that $A_{1,2}$ contains $\begin{psmallmatrix}
 1 & * \\
0 & 1 
\end{psmallmatrix}$ (since $0'$ corresponds to the matrix entry where $A_{1,2}[i,j] = A[i,j]= 0$), a contradiction. Similarly, if $\widetilde{A}$ contains $\begin{psmallmatrix}
 0^* & * \\
0' & 0^* 
\end{psmallmatrix}$, then $A_{1,2}$ contains $\begin{psmallmatrix}
 1 & * \\
0 & 1 
\end{psmallmatrix}$, a contradiction. 
\end{proof}

Next, we show an algorithm to sort the columns to exclude $\begin{psmallmatrix} 1 & 0^* \end{psmallmatrix}$ as a submatrix. This ordering will give us $A_3$, such that $A = A_{1,2} \odot A_3$, since $0^*$ are the zeroes in $A$ not in $A_{1,2}$

\begin{algorithm}[hbt!]
\caption{Reordering the columns of $\widetilde{A}$ }\label{alg:U_order}
  \DontPrintSemicolon
    $L_3 \gets \emptyset$\;
    \For{column $j \gets 1$ \KwTo $|V|$}{
        Add $j$ to the end of $L_3$ \;
        Let $S_j = \{i: \widetilde{A}[i, j]= 0^*\}$ \;
        \If{$S_j \neq \emptyset$} {
            Let $k_j$ be the leftmost column according to the ordering $L_3$ such that $\widetilde{A}[i, k_j] = 1$ for some $i \in S_j$ \; 
            In $L_3$ move $j$ to the left of $k_j$ \;
        }
    }
\end{algorithm}

\begin{lemma} \label{direction2}
The biadjacency matrix $\widetilde{A}$ with columns sorted according to $L_3$ is $\begin{psmallmatrix} 1 & 0^* \end{psmallmatrix}$-free.
\end{lemma}
\begin{proof}
$L_3$ is initially an empty list. In each iteration, we add a column into $L_3$ according to the algorithm. 
Denote by $A^j$ the submatrix of $\widetilde{A}$ with columns appearing in the list $L_3$ and ordered according to $L_3$ after finishing iteration $j$; note that $A^j$ contains all the rows ordered in the same way as $\widetilde{A}$. So, $A^j$ is an $n$-by-$j$ matrix. 

We will prove the following statement by induction, which would imply the lemma. 
\begin{quote}
For all $j =1, \ldots, n$, the matrix $A^j$ is $\begin{psmallmatrix} 1 & 0^* \end{psmallmatrix}$-free.    
\end{quote}

When $j=1$ the statement trivially holds. 

Consider the end of iteration $j$ where we know the statement holds for $A^{j-1}$. 
If $S_j = \emptyset$ (meaning that column $j$ does not contain $0^*$), adding this column to the end of $A^{j-1}$ cannot create the pattern $\begin{psmallmatrix} 1 & 0^* \end{psmallmatrix}$, so the statement also holds for $A^j$.

Next, consider the case when $S_j \neq \emptyset$. Assume by contradiction that $A^{j}$ contains $P=\begin{psmallmatrix} 1 & 0^* \end{psmallmatrix}$. Suppose that the submatrix of $A^j$ at row $i$ and columns $\{\ell, j\}$ matches with $P$. 
It is impossible that $\ell$ is to the left of $j$ in $A^j$: Since $\widetilde{A}[i,\ell] = 1$ and $\widetilde{A}[i,j] = 0^*$, column $j$ must be moved to the left of $\ell$ according to the algorithm. 

So we must have column $\ell$ to the right of $j$ in $A^j$, so $\widetilde{A}[i,\ell] =0^*$ and $\widetilde{A}[i,j] = 1$. Let $i_{min} \in S_j$ be a row such that $\widetilde{A}[i_{min}, j] = 0^*$ and $\widetilde{A}[i_{min}, k_j] = 1$. 
Let $i_{min} \in S_j$ be a row such that $\widetilde{A}[i_{min}, j] = 0^*$ and $\widetilde{A}[i_{min}, k_j] = 1$.
The algorithm guarantees that column $j$ is to the left of $k_j$ in $A^j$. 



Note that $k_j \neq \ell$, because otherwise the submatrix of $\widetilde{A}$ at columns $\{\ell, j\}$ and rows $\{i, i_{min}\}$ matches with the pattern $\begin{psmallmatrix}
 0^* & 1 \\
1 & 0^* 
\end{psmallmatrix}$ or $\begin{psmallmatrix}
 1 & 0^* \\
0^* & 1 
\end{psmallmatrix}$ 
depending on the original column and row orders. 
Therefore, it must be the case that $\{j, k_j, \ell\}$ are three distinct columns. 
Moreover, they appear in matrix $A^j$ in this order. 

Let us consider the cases on the values of $\widetilde{A}[i_{\min}, \ell]$ and $\widetilde{A}[i,k_j]$. 

\begin{itemize}
    \item Case 1: $\widetilde{A}[i_{min}, \ell] = 1$, $\widetilde{A}[i, k_j] = 0^*$. This case implies that the submatrix of $\widetilde{A}$ at rows $\{i, i_{\min}\}$ and columns $\{j, k_j\}$ matches with a forbidden pattern, a contradiction to~\Cref{A_free_matrices}. 

    \item Case 2: $\widetilde{A}[i_{min}, \ell] = 0^*$. This case implies that the submatrix of $A^{j-1}$ at row $i_{\min}$ and columns $\{k_j, \ell\}$ matches with pattern $P$, a contradiction to the induction hypothesis. 

    \item Case 3: $\widetilde{A}[i, k_j] = 1$. This implies that the submatrix of $A^{j-1}$ at row $i$ and columns $\{k_j, \ell\}$ matches with $P$, again a contradiction. 

    \item Case 4: $\widetilde{A}[i, k_j] = \widetilde{A}[i_{min}, \ell] = 0'$. 
    We consider two subcases based on the relative positions of rows $i$ and $i_{\min}$. 
    If $i < i_{min}$ then, the submatrix of $\widetilde{A}$ at rows $\{i, i_{\min}$ and columns $\{\ell, j\}$ matches a forbidden submatrix (\Cref{A_free_matrices}) (we recall that $\ell$ is to the left of $j$ in $\widetilde{A}$.) If $i_{min} < i $, the submatrix of $\widetilde{A}$ at rows $\{i, i_{\min}\}$ and columns $\{k_j, j\}$ match again a forbidden pattern. See~\Cref{fig:columns_moved_row_order} for an illustration of this case. 
\end{itemize}
Since all cases lead to a contradiction, we complete the inductive proof. 
\end{proof}

\begin{figure}[h]
\centering
\includegraphics[scale=0.24]{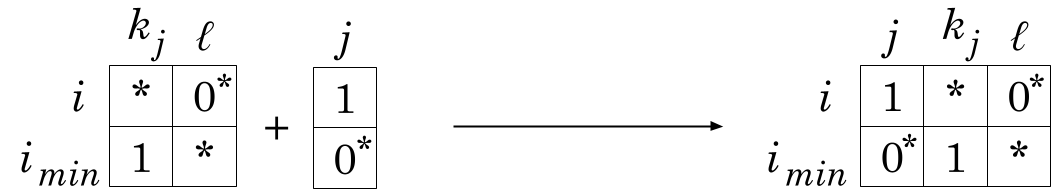}
\caption{On the left, we have $j$ yet to be inserted into the order, and on the right, the order $L$ is depicted.}
\label{fig:columns_moved}

\includegraphics[scale=0.24]{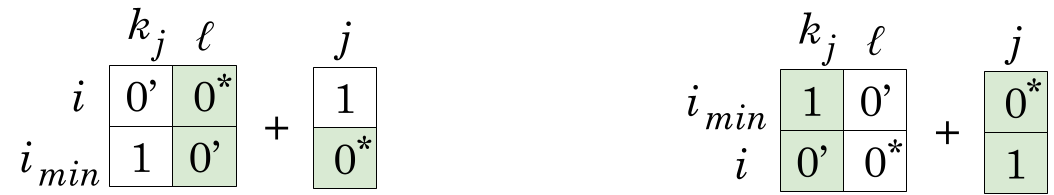}
\caption{Depicted are the forbidden submatrices with row orders $i < i_{min}$ and $i_{min} < i $ on the left and right respectively. }
\label{fig:columns_moved_row_order}
\end{figure}

\begin{lemma}
    There exists a chain graph $G_3\in \chain$ with the biadjacency matrix $A_3$ (and its extension $\widetilde{A}_3$)
    based on $L_3$, such that $A = A_3 \odot A_{1,2}$.
\end{lemma}
\begin{proof}
Consider the matrix $A'$ obtained from $\widetilde{A}$ by reordering the columns  according to $L_3$  in~\Cref{alg:U_order}. To avoid confusion with the rows and columns of $A$, we refer to rows and columns of $A'$ by vertices in $U$ and $V$ respectively. 
We construct a chain graph $G_3$ by ordering points on the line from left to right where each point $p_v$ corresponds to a column $v \in V$. For each row $u \in U$, we have a rightward ray that starts between column $v$ and $w$ where $w$ is the leftmost column for which $A'[u,w] = 1$. 
Since we can construct $G_3$ as an intersection graph of points and rays,  $G_3 \in \chain$. Let $A_3$ be the biadjacency matrix of $G_3$ where rows and columns are ordered according to $A$ (notice that this is a different order from $A'$). 
Clearly, $A \leq A_3$ by construction.

Now, we will show that $A = A_3 \odot A_{1,2}$. 
Since $A_{1,2}, A_3 \geq A$, we only need to argue that every zero in $A$ corresponds to a zero in  $A_{1,2}$ or $A_3$. 
Let us consider $A[i,j] = 0$.  
The first case is when $\widetilde{A}[i,j] = 0'$. In this case, $A_{1,2}[i,j] = 0$ and we are done. 

The other case is when $\widetilde{A}[i,j] = 0^*$. Assume that $A_3[i,j] = 1$ (for contradiction).
Let $u \in U$ and $v \in V$ be vertices corresponding to row $i$ and column $j$ respectively. 
Let $v_{\min}$ be the leftmost column for which $A'[u, v_{\min}] = 1$. It must be the case that $v_{\min}$ is to the left of $v$ in $A'$. Since $A'[u, v] = 0^*$, matrix $A'$ contains pattern $\begin{psmallmatrix} 1 & 0^* \end{psmallmatrix}$, a contradiction.

We conclude that $A = A_3 \odot A_{1,2}$. 
\end{proof}

In summary, we have constructed $G_{1,2} \in \chain^2$ and $G_3 \in \chain$ such that their biadjacency matrices satisfy $A = A_{1,2} \odot A_3$. This implies that $G \in \chain^3$, concluding the proof of the ``if'' direction.

\bibliography{ref}

\appendix

\end{document}